\documentclass[12pt]{amsart}
\usepackage{amscd,amsmath,amsthm,amssymb}
\usepackage[left]{lineno}
\usepackage{color}
\usepackage{stmaryrd}
\usepackage{graphicx}
\definecolor{verylight}{gray}{0.97}
\definecolor{light}{gray}{0.9}
\definecolor{medium}{gray}{0.85}
\definecolor{dark}{gray}{0.6}
 \def\NZQ{\mathbb}               
 
 \def\QQ{{\NZQ Q}}
 \def\ZZ{{\NZQ Z}}
 \def\RR{{\NZQ R}}

 \def\G{{\mathcal G}}

 \def\Pc{{\mathcal P}}
  \def\Qc{{\mathcal Q}}

    \def\Ac{{\mathcal A}}
 \def\Bc{{\mathcal B}}
 \def\ab{{\mathbf a}}
 
 \def\xb{{\mathbf x}}

 \def\opn#1#2{\def#1{\operatorname{#2}}} 
 \opn\chara{char} \opn\length{\ell} \opn\pd{pd} \opn\rk{rk}
 \opn\projdim{proj\,dim} \opn\injdim{inj\,dim} \opn\rank{rank}
 \opn\depth{depth} \opn\grade{grade} \opn\height{height}
 \opn\embdim{emb\,dim} \opn\codim{codim}
 
 \opn\Tr{Tr} \opn\bigrank{big\,rank}
 \opn\superheight{superheight}\opn\lcm{lcm}
 \opn\trdeg{tr\,deg}
 \opn\reg{reg} \opn\lreg{lreg} \opn\ini{in} \opn\lpd{lpd}
 \opn\size{size} \opn\sdepth{sdepth}
 \opn\link{link}\opn\fdepth{fdepth}\opn\lex{lex}
 \opn\tr{tr}
 \opn\type{type}
 \opn\gap{gap}
 \opn\arithdeg{arith-deg}
 \opn\astab{astab}
  \opn\dstab{dstab}
  \opn\pol{pol}
  \opn\mat{mat}
  \opn\indmat{indmat}
 \opn\div{div} \opn\Div{Div} \opn\cl{cl} \opn\Cl{Cl}
 \opn\Spec{Spec} \opn\Supp{Supp} \opn\supp{supp} \opn\Sing{Sing}
 \opn\Ass{Ass} \opn\Min{Min}\opn\Mon{Mon}
 \opn\Ann{Ann} \opn\Rad{Rad} \opn\Soc{Soc}
 \opn\Im{Im} \opn\Ker{Ker} \opn\Coker{Coker} \opn\Am{Am}
 \opn\Hom{Hom} \opn\Tor{Tor} \opn\Ext{Ext} \opn\End{End}
 \opn\Aut{Aut} \opn\id{id}
 
 \opn\nat{nat}
 \opn\pff{pf}
 \opn\Pf{Pf} \opn\GL{GL} \opn\SL{SL} \opn\mod{mod} \opn\ord{ord}
 \opn\Gin{Gin} \opn\Hilb{Hilb}\opn\sort{sort}
 \opn\PF{PF}\opn\Ap{Ap}
 \opn\mult{mult}
 \opn\aff{aff}
 \opn\relint{relint} \opn\st{st}
 \opn\lk{lk} \opn\cn{cn} \opn\core{core} \opn\vol{vol}  \opn\inp{inp} \opn\nilpot{nilpot}
 \opn\link{link} \opn\star{star}\opn\lex{lex}\opn\set{set}
 \opn\width{wd}
 \opn\Fr{F}
 \opn\QF{QF}
 \opn\G{G}
 \opn\type{type}\opn\res{res}
 \opn\conv{conv}
 \opn\Cl{Cl}
 \opn\Fitt{Fitt}
 \opn\gr{gr}
 
 \def\pot#1#2{#1[\kern-0.28ex[#2]\kern-0.28ex]}

 \opn\dirlim{\underrightarrow{\lim}}
 \opn\inivlim{\underleftarrow{\lim}}
 \let\sect=\cap
 \let\dirsum=\oplus
 
 \let\iso=\cong
 
 \let\Sect=\bigcap

 \let\to=\rightarrow
 \let\To=\longrightarrow
 \def\Implies{\ifmmode\Longrightarrow \else
         \unskip${}\Longrightarrow{}$\ignorespaces\fi}
 \def\implies{\ifmmode\Rightarrow \else
         \unskip${}\Rightarrow{}$\ignorespaces\fi}
 \def\iff{\ifmmode\Longleftrightarrow \else
         \unskip${}\Longleftrightarrow{}$\ignorespaces\fi}

 \let\:=\colon
 \newtheorem{Theorem}{Theorem}[section]
 \newtheorem{Lemma}[Theorem]{Lemma}
 \newtheorem{Corollary}[Theorem]{Corollary}
 \newtheorem{Proposition}[Theorem]{Proposition}

 \newtheorem{Example}[Theorem]{Example}

 \let\epsilon\varepsilon
 \let\kappa=\varkappa
 \def\qed{\ifhmode\textqed\fi
       \ifmmode\ifinner\quad\qedsymbol\else\dispqed\fi\fi}
 \def\textqed{\unskip\nobreak\penalty50
        \hskip2em\hbox{}\nobreak\hfil\qedsymbol
        \parfillskip=0pt \finalhyphendemerits=0}
 \def\dispqed{\rlap{\qquad\qedsymbol}}
 
 \opn\dis{dis}
 \def\pnt{{\raise0.5mm\hbox{\large\bf.}}}
 
 \opn\Lex{Lex}
\begin{document}

\title {The reduced divisor class group and the torsion number}

\author {J\"urgen Herzog and Takayuki Hibi}

\address{J\"urgen Herzog, Fachbereich Mathematik, Universit\"at Duisburg-Essen, Campus Essen, 45117
Essen, Germany} \email{juergen.herzog@uni-essen.de}

\address{Takayuki Hibi, Department of Pure and Applied Mathematics, Graduate School of Information Science and Technology,
Osaka University, Suita, Osaka 565-0871, Japan}
\email{hibi@math.sci.osaka-u.ac.jp}

\dedicatory{ }

\begin{abstract}
The reduced divisor class group of a normal Cohen--Macaulay graded domain together with its torsion number is introduced.  They are studied in detail especially for normal affine semigroup rings.
\end{abstract}

\thanks{ The second author was partially supported by JSPS KAKENHI 19H00637}




\maketitle

\setcounter{tocdepth}{1}
\section*{Introduction}
Let $P$ be a finite partially ordered set and $R$ the normal affine semigroup ring introduced in \cite{Hibi_ring}.  Nowadays  authors call $R$ the {\em Hibi ring}, but in the present paper we  call $R$ the {\em join-meet ring} arising from $P$, because its relations are given by the joins and meets of the distributive lattice defined by $P$.  It is shown \cite{HHN} that the divisor class group $\Cl(R)$ of $R$ is free of rank $p + q + e - d - 1$, where $p$ is the number of minimal elements of $P$, $q$ is the number of maximal elements of $P$, $e$ is the number of edges of the Hasse diagram of $P$ and $d = |P|$.  On the other hand, in \cite{Hibi_ring}, by studying the generators of the canonical module $\omega_R$ of $R$, it is proved that $R$ is Gorenstein if and only if $R$ is pure, i.e., every maximal chain of $P$ has the same cardinality.  In general, it is known that $R$ is Gorenstein if and only if the canonical class $[\omega_R]$ of $R$ is equal to $0$ in $\Cl(R)$.  In other words, $[\omega_R] = 0$ in $\Cl(R)$ if and only if $P$ is pure.  It is reasonable to ask how to compute $[\omega_R]$ in terms of combinatorics of $P$.  This natural question is what motivated the authors to write this paper in the first place.  Its satisfied solution will be given in Section $2$.

Let $R$ be a Noetherian local ring or a finitely generated graded $K$-algebra for which $R$ is a normal Cohen--Macaulay domain with a canonical module $\omega_R$.  In the first half of Section $1$, the new concepts, the  {\em reduced divisor class group} of $R$ and the {\em torsion number} of $R$, are introduced.    The reduced divisor class group of $R$ is $\overline{\Cl}(R)=\Cl(R)/\ZZ[\omega_R]$ and the torsion number of $R$ is the nonnegative integer $d(R)$ defined as follows:  let $\Fitt_i(G)$ denote the $i$th Fitting ideal of a finite Abelian group, and let   $r= \rank \overline{\Cl}(R)$. If $\Fitt_r(\Cl(R))=\Fitt_r(\overline{\Cl}(R))$,  then we set $d(R)=0$. Otherwise,   $d(R)$ is given by the identity
$
\Fitt_{r}(\overline{\Cl}(R))=(d(R)).
$
One has $d(R)=0$ if and only if $R$ is Gorenstein (Lemma \ref{zero}).  When $\Cl(R)$ is free of rank $r$, the torsion number $d(R)$ has a concrete interpretation.  In fact, one has $\overline{\Cl}(R)\iso \ZZ^{r-1}\dirsum \ZZ/(d(R))$ and $[\omega_R]$ is part of a basis of $\Cl(R)$ if and only if $d(R)=1$ (Lemma \ref{free}).  When $S \subset \ZZ^n$ is a normal affine semigroup, the divisor class group of the associated normal semigroup ring $R= K[S]$ is well understood.  In the latter half of Section $1$, the basic facts related to the divisor class group $\Cl(K[S])$ of $R= K[S]$, especially the result by Chouinard \cite{Ch} on a set of generating relations of $\Cl(R)$ are summarized in short.

Section $2$ will be devoted to the study of the divisor class groups of the join-meet ring of a finite partially ordered set.  As was discussed in \cite{HHN}, the information of the facets of the cone coming from $P$  (Stanley \cite{Stanley}) yields the relation matrix of $\Cl(R)$ and it gives the explicit expression of $[\omega_R]$ in terms of the basis of $\Cl(R)$, which is the satisfied solution of the original question as well as which directly explains why $[\omega_R] = 0$ in $\Cl(R)$ if and only if $P$ is pure (Theorem \ref{canonical_class}).

On the other hand, the detailed study of torsion numbers is achieved in Section $3$.  In the join-meet ring $R$, the torsion number can be an arbitrary nonnegative integer (Example \ref{hotelherzog}).  Furthermore, if a join-meet ring $R$ is nearly Gorenstein but not Gorenstein, then one has $d(R)=1$ (Corollary \ref{near}).  However, in general, even though a normal affine semigroup ring is nearly Gorenstein but not Gorenstein, it happens that $d(R) > 1$ (Example \ref{Essen}).

\section{The canonical class  and the torsion number}
\label{1}
Let $R$ be a Noetherian local ring or a  finitely generated graded $K$-algebra. We furthermore assume that $R$ is a normal Cohen-Macaulay domain with a  canonical module $\omega_R$. The canonical module   can be identified with a divisorial ideal. Let $\Cl(R)$ be the  divisor class group of $R$. The class of a divisorial ideal $I$ of $R$  will be denoted by $[I]$.  We choose  of  system a of generators $g_1,\ldots, g_m$ of $\Cl(R)$.  Then $[\omega]$  can be written as a linear combination  of these generators, say, $[\omega_R] = \sum_{i=1}a_ig_i$. The integer coefficients of this presentation depend of course on the choice of the generators. Of special interest is the case that $[\omega_R]=0$, because this is the case if and only if $R$ is Gorenstein. However the above linear combination does not tell us immediately, whether of not $[\omega_R]=0$. Thus we are looking for a more intrinsic invariant of the canonical class. To this end, we consider the group  $\overline{\Cl}(R)=\Cl(R)/\ZZ[\omega_R]$, and a certain Fitting ideal of it. We call $\overline{\Cl}(R)$ the {\em reduced divisor class group} of $R$.

Let us briefly recall the concept of Fitting ideals and their  basic properties. Let $M$ be a finitely generated module  over a commutative ring $R$ with generators $u_1,\ldots,u_n$ and with a relation matrix $A=[a_{ij}]_{i=1,\ldots, n\atop j= 1,\ldots,m}$. In other words, $\sum_{i=1,\ldots,n}a_{ij}m_i=0$ for all $j$, and these are the generating relations of $M$ with respect to these generators.
Given these data, the $i$th Fitting ideal $\Fitt_i(M)$ of $M$ is the
ideal $I_{n-i}(A)$ of $(n-i)$-minors of $A$. The Fitting ideals are invariants of the module, that is, they do not depend on the  choice of the system of generators and the relation matrix. One has $\Fitt_0(M)\subseteq \Fitt_1(M)\subseteq \cdots \subseteq \Fitt_n(M)=R$. If $R$ is a  domain, then $\rank M=\min\{i\: \Fitt_i(M)\neq 0\}$. Moreover, $M$ is free of rank $r$ if and only if $\Fitt_i(M)=0$ for $i<r$ and $\Fitt_r(M)=R$.

We may view any finitely generated Abelian group $G$ as a $\ZZ$-module, and  hence the Fitting ideals of $G$ are defined. Suppose $G$ has $n$ generators and the relation matrix $A$ has rank $m$. Then there exists an exact sequence $0\to \ZZ^m\to \ZZ^n\to G\to 0$, which implies that $\rank G=n-m$. Thus, if $r=\rank G$, then $r$ is the smallest integer for which $\Fitt_r(G)\neq 0$.

Now we are ready to define the {\em  torsion number} $d(R)$ of $R$. Let  $r= \rank \overline{\Cl}(R)$. If $\Fitt_r(\Cl(R))=\Fitt_r(\overline{\Cl}(R))$,  then we set $d(R)=0$. Otherwise,   $d(R)$ is given by the identity
\[
\Fitt_{r}(\overline{\Cl}(R))=(d(R)).
\]
We have
\begin{Lemma}
\label{zero}
$R$ is Gorenstein if and only if $d(R)=0$.
\end{Lemma}

\begin{proof}
Suppose that $R$ is Gorenstein.  Then $\overline{\Cl}(R)=\Cl(R)$, and so $\Fitt_{r}(\overline{\Cl}(R))=\Fitt_{r}(\Cl(R))$.

Conversely, suppose that $\Fitt_{r}(\overline{\Cl}(R))=\Fitt_{r}(\Cl(R))$. Let $s=\rank \Cl(R)$. Then $s\geq r\geq s-1$.  Suppose $r=s-1$. Then  $\Fitt_{s-1}(\Cl(R))=\Fitt_{r}(\overline{\Cl}(R))\neq 0$, a contradiction. Hence $\rank \Cl(R)=\rank \overline{\Cl}(R)$,
and $\Cl(R)\iso \ZZ^r\dirsum H$,where $H$ is a finite group. Since  $\rank \Cl(R)=\rank \overline{\Cl}(R)$, it follows that $[\omega_R]\in H$. Therefore, $\overline{\Cl}(R)\iso \ZZ^r\dirsum \overline{H}$, where $\overline{H}=H/\ZZ[\omega_R]$.  It follows that $|H|=\Fitt_r(\Cl(R) )= \Fitt_r(\overline{\Cl}(R) )=|\overline{H}|$. Therefore,  $H=\overline{H}$. This implies that $[\omega_R]=0$.
\end{proof}

When the divisor class group is free, then $d(R)$ has a concrete interpretation.

\begin{Lemma}
\label{free}
Suppose $\Cl(R)$ is free of rank $r$. Then $\Cl(R)\iso \ZZ^r$. Under this isomorphism,  let $[\omega_R]=(a_1,\ldots,a_r)$ with $a_i\in \ZZ$. Then $d(R)=\gcd(a_1,\ldots,a_r)$.  In particular, $\overline{\Cl}(R)\iso \ZZ^{r-1}\dirsum \ZZ/(d(R))$ and $[\omega_R]$ is part of a basis   of  $\Cl(R)$ if and only if $d(R)=1$,
\end{Lemma}

\begin{proof}
With respect to the basis of $\Cl(R)$ corresponding to the isomorphism $\Cl(R)\iso \ZZ^r$,   the relation matrix of $\overline{\Cl}(R)$ is given by $[a_1,\ldots, a_r]$.  We have $[\omega_R]=0$,  if and only if all $a_i=0$, and this is the case if and only if $ \rank\overline{\Cl}(R)=r$. In this case, $\Fitt_{r}(\overline{\Cl}(R))=\Fitt_{r}(\Cl(R))( =\ZZ)$,  and hence $d(R)=0$ according to our definition. On the other hand, if  $a_i\neq 0$  for some $i$, then $\rank \overline{\Cl}(R) =r-1$ and $\Fitt_{r-1}( \overline{\Cl}(R))=(\gcd(a_1,\ldots,a_r))$. This yields the statements of the lemma.
\end{proof}

Let $K$ be a field. For a normal affine semigroup $S\subset \ZZ^n$  the  divisor class group of the associated semigroup ring $R= K[S]$ is well understood. We use the notation introduced  in \cite{BH} and denote by  $\ZZ S$ the smallest subgroup of $\ZZ^{n}$  containing $S$ and by $\RR_+ S\subset \RR^{n}$ the smallest cone containing $S$.   Since $R$ is normal,  Gordon's lemma \cite[Proposition 6.1.2]{BH} guaranties that $S=\ZZ^{n}\sect \RR_+ S$.  After a suitable change of  coordinates, one may always assume that $\ZZ S=\ZZ^{n}$.
 Notice that $\RR_+ S\subset \ZZ^ n$  is a positive rational cone. Given any such cone $C$, one has that $\ZZ^n\sect C$ is a normal affine semigroup.
Let $H_1,\ldots,H_r$ be the supporting hyperplanes of $C$. Since for each $i$, the hyperplane   $H_i$ is spanned by lattice points, a  linear form  $f_i=\sum_{i=1}^{n}a_{ij}x_j$ defining $H_i$ has rational coefficients. By clearing denominators we may assume that all $a_{ij}$ are integers, and then dividing $f_i$ by the greatest common divisor of the $a_{ij}$, we may furthermore assume that $\gcd(a_{i1},\ldots,a_{in})=1$. Up to sign,  this linear form $f_i$ is uniquely determined by $H_i$. Let $p$ be a lattice point in the relative  interior of $C$. By replacing $f_i$ by $-f_i$,  if necessary, we may assume that $f_i(p)>0$ for all $i$. We call this normalized uniquely determined  linear form $f_i$ the {\em support form} of $H_i$.

We recall the following facts:

(i) Let $P_i\subset R$ be the $K$ subvector space of $K[S]$ spanned by all monomials $\xb^\ab$ with $\ab\in C\setminus H_i$. Then $P_i$ is a monomial prime ideal of height 1, and we have  $\{P_1,\ldots,P_r\}$ is the set of all monomial prime ideals of height 1 in $R$.

(ii) $\Cl(R)$ is generated by the classes $[P_1], \ldots,  [P_r]$.

(iii)  (Chouinard \cite{Ch})   $\sum_{i=1}^ra_{ij}[P_i]=0$  for $j=1,\ldots,n$, and this  is  a set of  generating relations of $\Cl(R)$. In other words, the $r\times n$-matrix
$A_R=[a_{ij}]_{i=1,\ldots,r\atop j=1,\ldots, n}$ is a  relation matrix of $\Cl(R)$. and we have an exact sequence of abelian groups
\[
0\To \ZZ^n\stackrel{A_R}{\To}\ZZ^r\To \Cl(R)\To 0.
\]
(iv) $\Cl(R)$ is free of rank $s$ if and only if $\Fitt_i(\Cl(R))=0$ for $i<s$  and $\Fitt_s(\Cl(R))=\ZZ$,   equivalently, if $I_{n-s}(A_R)=\ZZ$ and $\rank A_R=n-s$.

By a theorem of Danilov and Stanley (see \cite[Theorem 6.3.5]{BH}),  $\omega_R$ is generated by the monomials $\xb^{\ab}$ for which $\ab$ belongs to the relative interior of $C$. This implies that $\omega_R=\Sect_{i=1}^rP_i$, and hence  $[\omega_R]= \sum_{i=1}^r[P_i]$. Consequently,  $\overline{\Cl}(R)$ has the relation matrix $\overline{A}_R$, where $\overline{A}_R$ is obtained from ${A}_R$ by adding a column whose entries are all one.

If $\Cl(R)$ is free of rank $r$, then $\rank \overline{\Cl}(R)=r-1$, and hence $d(R)$ is the generator of the principal  ideal $\Fitt_{r-1}(\overline{\Cl}(R))=I_{n-r+1}(\overline{A}_R)$.

\section{Divisor class groups of join-meet rings}
\label{2}
The present section will be devoted to the discussion of the divisor class group of the normal semigroup ring, introduced in \cite{Hibi_ring}, arising from a finite partially ordered set.  Let $P = \{x_1, \ldots, x_n\}$ be a finite partially ordered set and suppose that that $i$ is smaller than $j$ whenever $x_i < x_j$ in $P$.  Let $\hat{P} = P \cup \{\hat{0}, \hat{1}\}$, where $\hat{0} < x_i < \hat{1}$ for $1 \leq i \leq n$.  Let $E(\hat{P})$ denote the set of edges of the Hasse diagram of $\hat{P}$.  Thus $(x,y) \in \hat{P}\times \hat{P}$ belongs to $E(\hat{P})$ if $x < y$ in $\hat{P}$ and $x < z < y$ for no $z \in \hat{P}$.  Following \cite[p.~10]{Stanley}, one associate each $e \in E(\hat{P})$ with the linear form $f_e$ by setting
\[
f_e
= \left\{
\begin{array}{ll}
x_i & \text{if} \, \, \, \, \, e = (x_i, \hat{1});\\
x_i - x_j& \text{if} \, \, \, \, \, e=(x_i, x_j) \in P \times P;\\
x_0 - x_{j}& \text{if} \, \, \, \, \, e = (\hat{0}, x_j).\\
\end{array}
\right.
\]
Let $C \subset \RR_+^{n+1}$ denote the cone whose supporting hyperplanes are those $H_e$ defined by $f_e$ with $e \in E(\hat{P})$.  Let $K$ be a field and $R = K[C \cap \ZZ^{n+1}]$ the affine semigroup ring, called the {\em join-meet ring} arising from $P$.  It is known \cite{Hibi_ring} that the the join-meet ring $R = K[C \cap \ZZ^{n+1}]$ is normal.  In particular, $R = K[C \cap \ZZ^{n+1}]$ is Cohen--Macaulay.  
The divisor class group $\Cl(R)$ of $R = K[C \cap \ZZ^{n+1}]$ is generated by the classes $[P_e]$ with $e \in E(\hat{P})$, where $P_e$ is the monomial prime ideal of height $1$ arising from $H_e$.  It is shown \cite{HHN} that $\Cl(R)$ is free of rank $|E(\hat{P})|-(n+1)$.

Following \cite{HHN} one fixes a spanning tree $T = \{e_0, \ldots, e_{n}\}$ of $E(\hat{P})$, where $e_i = (x_i, x_{i'})$ with $x_{0} = \hat{1}$.
Let $E(\hat{P}) = \{e_0, \ldots, e_{n}, e_{n+1}, \ldots, e_r\}$.  Let $A_R=[a_{ij}]_{i=0,\ldots,r \atop j=0,\ldots, n}$ denote the relation matrix of $\Cl(R)$, where $a_{ij}$ is the coefficient of $x_j$ in $f_{e_i}$.  The choice of the tree $T$ says that the submatrix of $A_R$ consisting of the first $n+1$ rows is an upper triangle matrix with each diagonal entry $1$.  It then follows that $[P_{n+1}], \ldots, [P_r]$ is a basis of the free abelian group $\Cl(R)$, where $P_i = P_{e_i}$.  In the divisor class group $\Cl(R)$, for each $0 \leq i \leq n$ one writes
\begin{eqnarray}
\label{expression}
[P_i] = \sum_{j=n+1}^{r} c_j^{(i)}[P_j], \, \, \, \, \, c_j^{(i)} \in \QQ.
\end{eqnarray}
Each $c_j^{(i)} \in \QQ$ can be computed as follow: For each edge $e_j = (x, y)$ with $n+1 \leq j \leq r$, the subgraph $G_j$ consisting of the edges $e_0, \ldots, e_n, e_j$ possesses a unique cycle $C_j$.  One fixes the orientation of $C_j$ with $x \to y$.  If $e_i = (x_i, x_{i'})$ with $0 \leq i \leq n$ appears in $C_j$ whose orientation is $x_i \to x_{i'}$, then one has $c_j^{(i)}=1$.  If $e_i = (x_i, x_{i'})$ with $0 \leq i \leq n$ appears in $C_j$ whose orientation is $x_{i'} \to x_i$, then one has $c_j^{(i)}=-1$.  If $e_i$ with $0 \leq i \leq n$ does not appear in $C_j$, then one has $c_j^{(i)}=0$.

One claims the validity of the above computation of $c_j^{(i)}$.  In other words, $[P_0], \ldots, [P_n]$ with the expression (\ref{expression}) together with $[P_{n+1}], \ldots, [P_r]$ could satisfy the relations of the columns of $A_R$.  Let $x_i \in P \cup \{\hat{0}\}$ with $\hat{0} = x_0$.  Let $\Ac$ denote the set of edges of $\hat{P}$ of the form $(x_i, x_{i'})$ and $\Bc$ that of the form $(x_{i''}, x_i)$.  If the cycle $C_j$, where $n+1 \leq j \leq r$ intersects $\Ac \cup \Bc$, then one of the followings occurs:

(i) $|C_j \cap \Ac| = |C_j \cap \Bc| = 1$;

(ii) $|C_j \cap \Ac| = 2$ and  $C_j \cap \Bc = \emptyset$;

(iii) $C_j \cap \Ac = \emptyset$ and $|C_j \cap \Bc| = 1$.

\noindent
In each of the above (i), (ii) and (iii), the total sum of $[P_{j}]$ appearing in $[P_e]$'s with $e \in \Ac$ is equal to that of $[P_{j}]$ appearing in $[P_e]$'s with $e \in \Bc$.  Hence $[P_0], \ldots, [P_n]$ with the expression (\ref{expression}) together with $[P_{n+1}], \ldots, [P_r]$ could satisfy the relations of the $i$\,th column of $A_R$, as desired.

\begin{Example}
{\em
Let $P = \{x_1,x_2,x_3,x_4,x_5,x_6\}$ be the finite partially ordered set of Figure $1$.
\begin{figure}
\includegraphics[width=34mm]{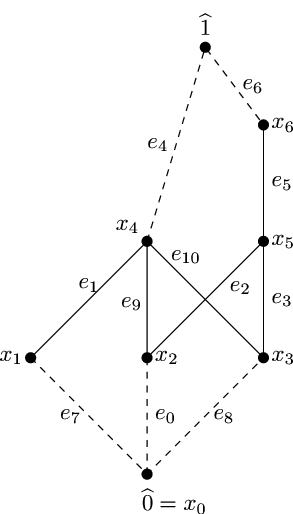}
\caption{poset $P$}
\end{figure}

\noindent
The tree $T = \{e_0,\ldots,e_6\}$ of Figure $2$ satisfies the above condition.
\begin{figure}
\includegraphics[width=34mm]{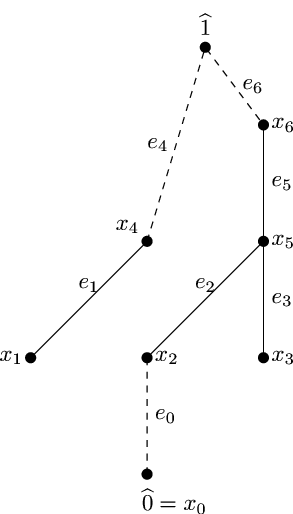}
\caption{tree $T$}
\end{figure}

\noindent
The cycle $C_7$ consists of the edges $e_7, e_1, e_4, e_6, e_5, e_2, e_0$ (Figure $3$).
\begin{figure}
\includegraphics[width=34mm]{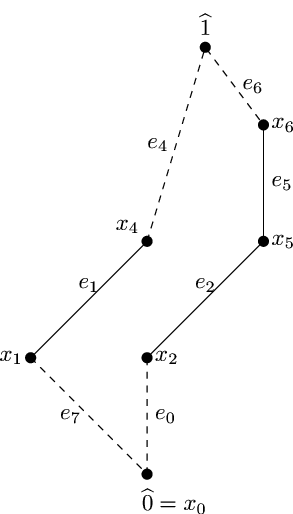}
\caption{cycle $C_7$}
\end{figure}

\noindent
Fix the orientation of $C_7$ with
\[
x_0 \to x_1 \to x_4 \to \hat{0} \to x_6 \to x_5 \to x_2 \to x_0.
\]
Thus the coefficient of $[P_7]$ in each of $[P_1], [P_4]$ is $1$, the coefficient of $[P_7]$ in each of $[P_0], [P_2], [P_5], [P_6]$ is $-1$ and the coefficient of $[P_7]$ in $[P_3]$ is $0$.  One has
\begin{align*}
[P_0] &= -[P_7] - [P_8],\\
[P_1] &= [P_7], \\
[P_2] &= -[P_7] - [P_8] - [P_9], \\
[P_3] &=   [P_8] - [P_{10}], \\
[P_4] &= [P_7] +  [P_9] + [P_{10}], \\
[P_5] &= -[P_7] - [P_9] - [P_{10}], \\
[P_6] &= -[P_7]  - [P_9] - [P_{10}].
\end{align*}
Thus in particular
\[
[\omega_R] = -[P_7]-[P_9]-[P_{10}].
\]
}
\end{Example}

Now, it is of interest to know  when $[\omega_R] = \sum_{e \in E(\hat{P})}[P_e] = 0$ in $\Cl(R)$, because this is the case if and only if $R$ is Gorenstein.

\begin{Theorem}
\label{canonical_class}
In $\Cl(R)$, one has $[\omega_R] = 0$ if and only if $P$ is pure.
\end{Theorem}

\begin{proof}  One employs the notation as above.

{\bf (``if'')}  Suppose that $P$ is pure.  Then, clearly, in each cycle $C_j$, the number of $e_i = (x_i, x_{i'})$ with $0 \leq i \leq n$ appearing in $C_j$ whose orientation is $x_i \to x_{i'}$ is exactly one less than that of $e_i = (x_i, x_{i'})$ with $0 \leq i \leq n$ appearing in $C_j$ whose orientation is $x_{i'} \to x_i$.  Hence each coefficient $q_j$ of $[\omega_R] = \sum_{j=n+1}^rq_i[P_j]$ is equal to $0$.

\medskip

\noindent
{\bf (``only if'')}  Suppose that $P$ is {\em not} pure and
\[
C : x < x_{i_1} < \cdots < x_{i_s} < y, \, \, \, \, \, \, \, \, \, \, C' : x < x_{i'_1} < \cdots < x_{i'_{s'}} < y
\]
are maximal chains of the interval $[x, y]$ of $\hat{P}$ with $s < s'$ for which $i_j \neq i'_{j'}$ for each $j$ and $j'$.  One can choose a tree $T$ which contains all edges except for $(x, x_{i'_1})$ appearing in the chains $C$ and $C'$.  Let $e_j = (x, x_{i'_1})$.  It then follows that the coefficient $q_j$ of $[\omega_R] = \sum_{j=n+1}^rq_j[P_j]$ is equal to $s' - s \neq 0$.  Hence $[\omega_R] \neq o$, as desired.
\end{proof}

Theorem \ref{canonical_class} gives an alternative proof to the old results that the join-meet ring $R = K[C \cap \ZZ^{n+1}]$ is Gorenstein if and only if $P$ is pure (\cite[p.~105]{Hibi_ring}).

\section{Computation of the torsion number}
\label{3}
Let $R$ be a normal Cohen--Macaulay domain with free divisor class group of rank $r$, and let $b_1,\ldots, b_r$ be a basis of $\Cl(R)$. Then $[\omega_R]=\sum_{i=1}^rc_ib_i$ with $c_i\in \ZZ$ for all $i$. Of course, a basis of $\Cl(R)$ is not uniquely determined. In Section $2$ we recalled that for given poset $P$ each spanning tree of $E(\hat{P})$ yields a basis of the class group of the associated join-meet ring.  For different bases the coefficients $c_i$ in the presentation of $[\omega_R]$ differ.  However $\gcd(c_1,\ldots,c_r)$ is independent of the choice of the basis, because it is just the torsion number $d(R)$ of $R$, defined in Section~\ref{1}.

\begin{Example}
\label{hotelherzog}
{\em $P$ be the poset with components $P_1$ and $P_2$ where $P_1$ and $P_2$ are chains of length $a$ and $b$, say, $P_1: x_0 < \cdots < x_a$ and $P_2: y_0 < \cdots < y_b$.  Fix the tree $T$ in $\hat{P}$ consisting of the edges belonging to $E(\hat{P}) \setminus (x_0,x_a)$, where $x_0 = \hat{0}$.  Then $[P_{e}]$ with $e = (x_0,x_a)$ is a basis of $\Cl(R)$.  The computation in Section $2$ yields $[P_{e'}] = [P_{e}]$ if $e' \in E(P_1) \cup \{(x_a, \hat{1})\}$ and $[P_{e''}] = -[P_{e}]$ if $e'' \in E(P_2) \cup \{(\hat{0}, y_1), (y_b, \hat{1})\}$.  Hence $[\omega_R]=(a-b)[P_e]$ and $d(R) = a-b$.
}
\end{Example}

The Example \ref{hotelherzog} shows that $d(R)$ can be any number. However, for any join-meet ring, the torsion number can be bounded as follow.

\begin{Proposition}
\label{bound}
Let $P$ be a finite poset.  Let $L_1:x_0 < \cdots < x_a$ and $L_2: y_0 < \cdots < y_b$ be maximal chains of $P$ for which $x_i \neq y_j$ for each $i$ and $j$.  Then $d(R)$ divides $a-b$.
\end{Proposition}

\begin{proof}
Fix the tree $T$ in $\hat{P}$ whose edges contains all edges belonging to
\[
E=E(L_1) \cup E(L_2) \cup \{(x_a,\hat{1}),(\hat{0},y_0),(y_b,\hat{1})\}.
\]
Then $e=(\hat{0},x_0) \not\in E(T)$.  The unique cycle in $T$ consists of the edges belonging to $E \cup \{e\}$.  Hence, as was done in Example \ref{hotelherzog}, the coefficient of $[P_e]$ of $[\omega_R]$ is equal to $a - b$.  Thus in particular $d(R)$ divides $a-b$, as desired.
\end{proof}

If $R$ is nearly Gorenstein but not Gorenstein, then one has $a - b = 1$ (\cite{nearlyGor}).  In particular, one has $d(R)=1$.

\begin{Corollary}
\label{near}
If the join-meet ring $R$ is nearly Gorenstein but not Gorenstein, then $d(R)=1$.
\end{Corollary}

Here is another example of a nearly Gorenstein ring which is not Gorenstein and whose torsion number is $1$.

\begin{Proposition}
Let $K$ be a field,  let $X$ be an $m\times n$-matrix of indeterminates with $m\leq n$, and let $R=K[X]/I_{r+1}(X)$. Then $\Cl(R)$ is free of rank $1$,  and if $R$ is nearly Gorenstein but not Gorenstein  then  $d(R)=1$.
\end{Proposition}

\begin{proof}
The class group of  $R$ is isomorphic to $[P]\ZZ\iso \ZZ$,   where $P$ is the prime ideal  in $R$ generated by the $r$-minors of the first $r$ rows $X$ modulo $I_{r+1}(X)$,  see \cite[Theorem 7.3.5]{BH}.  Furthermore, $\omega_R=P^{(n-m)}$, see \cite[Theorem7.3.6]{BH}.

In \cite[Theorem 1.1]{FHST} it is shown that  $\tr(\omega_R)=I_{r}(X)^{n-m}R$. From  this fact it follows that $R$  is nearly Gorenstein  but not Gorenstein if and only if $r=1$ and $n-m=1$, and that in this case $[\omega_R]=[P]$. This implies that $d(R)=1$.
\end{proof}

One would expect that  torsion number, if defined, is always $1$ for rings which are nearly Gorenstein but not Gorenstein.  However, the following family of examples show that this is not the case.

\begin{Example}
\label{Essen}
{\em
Let $R_m = K[x_1, \ldots, x_m]$ denote the polynomial ring in $m$ variables over a field $K$ and $S_n = K[y_1, \ldots, y_n]$ that in $n$ variables over $K$.  Let $R_m^{(p)}$, where $1 \leq p \in \ZZ$, be the $p$\,th Veronese subring of $R_m$.  It is known that $R_m^{(p)}$ is normal and Cohen--Macaulay (\cite[p.~193]{GW}).  Furthermore, $R_m^{(p)}$ is Gorenstein if and only if $p$ divides $m$ (\cite{Mat})).  Fix positive integers $m, n, p$ and $q$ and write $R = R_m^{(p)}\#S_n^{(q)}$ for the Segre product of $R_m^{(p)}$ and $S_n^{(q)}$.

Let $\Pc \subset \RR^{m+n}$ denote the convex polytope consisting of those $$(a_1, \ldots, a_m, b_1, \ldots, b_n) \in \RR^{m+n}$$ for which

(i) $a_i \geq 0$ for $1 \leq i \leq m$;

(ii) $b_j \geq 0$ for $1 \leq j \leq n$;

(iii) $\sum_{i=1}^{m} a_i = p$;

(iv) $\sum_{j=1}^{n} b_j = q$.

\noindent
As is discussed in \cite{DH}, the convex polytope $\Pc$ is a lattice polytope of dimension $m+n-2$.  (A convex polytope is called a {\em lattice polytope} if each of the vertices has integer coordinates.)  The Segre product $R = R_m^{(p)}\#S_n^{(q)}$ is the toric ring of $\Pc$.  In other words, $R$ is generated by those monomials $$\left(\prod_{i=1}^{m} x_i^{a_i}\right)\left(\prod_{j=1}^{n} y_i^{a_i}\right)$$ with $(a_1, \ldots, a_m, b_1, \ldots, b_n) \in \Pc \cap \ZZ^{n+m}$.  Furthermore, $R$ is normal and Cohen--Macaulay (\cite[p.~198]{GW}).  Now, one introduces the lattice polytope $\Qc \subset \RR^{m+n-2}$ of dimension $m+n-2$ consisting of those $$(a_1, \ldots, a_{m-1}, b_1, \ldots, b_{n-1}) \in \RR^{m+n-2}$$ for which

(i) $a_i \geq 0$ for $1 \leq i \leq m-1$;

(ii) $b_j \geq 0$ for $1 \leq j \leq n-1$;

(iii) $\sum_{i=1}^{m-1} a_i \leq p$;

(iv) $\sum_{j=1}^{n-1} b_j \leq q$.

\noindent
The facets of $\Qc$ are

(i) $x_i = 0$ for $1 \leq i \leq m-1$;

(ii) $y_j = 0$ for $1 \leq j \leq n-1$;

(iii) $\sum_{i=1}^{m-1} x_i = p$;

(iv) $\sum_{j=1}^{n-1} y_j = q$.

\noindent
One can regard the Segre product $R$ to be the toric ring of $\Qc$.  Let $C \subset \RR_+^{m+n+1}$ denote the cone whose supporting hyperplanes are

(i) $H_i : x_i = 0$ for $1 \leq i \leq m-1$;

(ii) $H'_j : y_j = 0$ for $1 \leq j \leq n-1$;

(iii) $H: -\sum_{i=1}^{m-1} x_i + pt = 0$;

(iv) $H': -\sum_{j=1}^{n-1} y_j + qt = 0$.

\noindent
Let $P_i$ denote the monomial prime ideal of height $1$ arising from $H_i$ and $Q_j$ that arising from $H'_j$.  Let $P$ denote the monomial prime ideal of height $1$ arising from $H$ and $Q$ that arising from $H'$.  The divisor class group $\Cl(R)$ is generated by
\[
[P_1], \ldots, [P_{m-1}], [P], [Q_1], \ldots, [Q_{n-1}], [Q]
\]
whose relations are
\[
[P_1] = \cdots = [P_{m-1}] = [P], \, \, [Q_1] = \cdots = [Q_{n-1}] = [Q], \, \, p[P] + q[Q] = 0.
\]
Hence
\[
\Cl(R) = (\ZZ[P] \bigoplus \ZZ[Q])/(p[P] + q[Q]).
\]
In particular one has $\Cl(R) =\ZZ$ if and only if $p$ and $q$ are relatively prime.  Since the canonical class is $[\omega_R] = m[P]+n[Q]$, it follows that $R$ is Gorenstein if and only if $(m,n) = c(p,q)$ for some integer $c > 1$.  In particular if $R$ is Gorenstein, then each of $R_m^{(p)}$ and $R_n^{(q)}$ is Gorenstein.  (See also \cite[chapter 4]{GW}.)  Furthermore, the Segre product $R$ is nearly Gorenstein, but not Gorenstein if and only if $p$ divides $m$, $q$ divides $n$ and $|m/p - n/q| = 1$ (\cite{nearlyGor}).  If $p$ and $q$ are relatively prime and if $p'$ and $q'$ are integers with $p'p + q'q = 1$, then $\Cl(R)$ is free of rank $1$ which is generated by $-q'[P] + p'[Q]$.

For example, $R = R_4^{(2)}\#S_9^{(3)}$ is nearly Gorenstein, but not Gorenstein and $\Cl(R)$ is free of rank $1$ which is generated by $[P] + 2[Q]$.  Since
\[
[\omega_R] = 4[P] + 9[Q] = - (2[P] + 3[Q]) + 6 ([P] + 2[Q]),
\]
one has $d(R) = 6$.
}
\end{Example}

Finally, we add an example of the computation of the torsion number of $R$ when $\Cl(R)$ is not free.

\begin{Example}
{\em
Let $K$ be a filed, and let $R=K[x_1,\ldots,x_n]^{(r)}$  be the $r$th Veronese subring of the polynomial ring $K[x_1,\ldots,x_n]$. Then the support forms of the hyperplanes describing the cone of  the natural embedding of the semigroup describing $R$ are $x_i\geq 0$ for $i=1,\ldots,n-1$  and $-(x_1+\cdots +x_{n-1})+rt\geq 0$. Therefore, we have
\[
\overline{A}_R=
\begin{bmatrix}
1& 0&\cdots&0 & 0& 1\\
0 & 1&\cdots  & 0& 0 & 1\\
\vdots & \vdots &\ddots &\vdots & \vdots &\vdots \\
0 & 0 & \cdots& 1 & 0 & 1\\
-1 & -1  &\cdots& -1 & r  &1
\end{bmatrix}
\]
The torsion number of $R$ is then given by $I_n(\overline{A}_R)=(r,n)$.  Therefore, $d(R)=(\gcd(r,n)$.

It is shown in \cite[Corollary 4.8]{nearlyGor}  that any Veronese subring of the polynomial is nearly Gorenstein, $d(R)$ can be any number.
}
\end{Example}

\end{document}